\documentclass[12pt, reqno]{amsart}
\UseRawInputEncoding
\usepackage{amsmath, amsthm, amscd, amsfonts, amssymb, graphicx, color}
\usepackage[bookmarksnumbered, colorlinks, plainpages]{hyperref}
\input{mathrsfs.sty}

\hypersetup{colorlinks=true,linkcolor=red, anchorcolor=green,
citecolor=cyan, urlcolor=red, filecolor=magenta, pdftoolbar=true}

\newtheorem{theorem}{Theorem}[section]
\newtheorem{lemma}[theorem]{Lemma}
\newtheorem{proposition}[theorem]{Proposition}
\newtheorem{corollary}[theorem]{Corollary}
\theoremstyle{definition}
\newtheorem{definition}[theorem]{Definition}
\newtheorem{example}[theorem]{Example}

\theoremstyle{remark}
\newtheorem{remark}[theorem]{Remark}
\newtheorem{notation}[theorem]{Notation}
\numberwithin{equation}{section}

\begin{document}

\setcounter{page}{1}

\title[$C^*$-module operators which satisfy in GCSI]
{$C^*$-module operators which satisfy in the generalized Cauchy--Schwarz type inequality}

\author[A.~Zamani]
{Ali Zamani}

\address{School of Mathematics and Computer Sciences, Damghan University, P.O.BOX 36715-364, Damghan, Iran}
\email{zamani.ali85@yahoo.com}

\subjclass[2010]{46L08; 47A63}

\keywords{Hilbert $C^*$-module, hyponormal operators, operator inequality, Cauchy--Schwarz inequality.}
\begin{abstract}
Let $\mathcal{L}(\mathscr{H})$ denote the $C^*$-algebra of adjointable operators on
a Hilbert $C^*$-module $\mathscr{H}$. We introduce the generalized Cauchy--Schwarz inequality
for operators in $\mathcal{L}(\mathscr{H})$ and investigate various properties of operators which
satisfy the generalized Cauchy--Schwarz inequality.
In particular, we prove that if an operator $A\in\mathcal{L}(\mathscr{H})$ satisfies the generalized Cauchy--Schwarz inequality
such that $A$ has the polar decomposition, then $A$ is paranormal.
In addition, we show that if for $A$ the equality holds in the generalized Cauchy--Schwarz inequality,
then $A$ is cohyponormal. Among other things, when $A$
has the polar decomposition, we prove that $A$ is semi-hyponormal if and only if
$\big\|\langle Ax, y\rangle\big\| \leq \big\|{|A|}^{1/2}x\big\|\big\|{|A|}^{1/2}y\big\|$
for all $x, y \in\mathscr{H}$.
\end{abstract} \maketitle
\section{Introduction and preliminaries}
One of the most important inequalities in mathematics is
the Cauchy--Schwarz inequality. Its most familiar version states that in a
Hilbert space $\big(X, [\cdot, \cdot]\big)$, it holds
\begin{align*}
|[\zeta, \eta]| \leq \|\zeta\|\|\eta\| \qquad (\zeta, \eta \in X).
\end{align*}
This classical inequality has a lot of elegant applications, for instance, in
classical and modern analysis, partial differential equations, multivariable calculus and probability theory.
There are also several extensions of this inequality in various settings for different objects,
see \cite{A.B.F.M} and references therein.
Some generalizations of the Cauchy–-Schwarz inequality in Hilbert spaces can be found in \cite{B.D, C.K.K, Wat}.

The notion of Hilbert $C^*$-module is a generalization of that of Hilbert space in which the inner product
takes its values in a $C^*$-algebra instead of the complex field.
We use $\mathcal{L}(\mathscr{H})$ to denote the $C^*$-algebra of adjointable operators
on a Hilbert $C^*$-module $\big(\mathscr{H}, \langle \cdot, \cdot\rangle\big)$.
Let $\mathcal{L}(\mathscr{H})_{+}$ be the set of positive elements in $\mathcal{L}(\mathscr{H})$.
For any $A\in\mathcal{L}(\mathscr{H})$, the range
and the null space of $A$ are denoted by $\mathcal{R}(A)$ and $\mathcal{N}(A)$, respectively.
An operator $A\in\mathcal{L}(\mathscr{H})$ is said to be normal if $AA^* = A^*A$, cohyponormal if $|A|^2 \leq |A^*|^2$,
semi-hyponormal if $|A^*| \leq |A|$ and paranormal if $\|Ax\|^2 \leq \|A^2x\|\|x\|$ for every $x \in \mathscr{H}$.
For details about $C^*$-algebras and Hilbert $C^*$-modules we refer the reader to \cite{Lan, Wegge}.

A version of the Cauchy--Schwarz inequality in $\mathscr{H}$ appeared in \cite{Lan} as follows:
\begin{align}\label{I.1.1}
\|\langle x, y\rangle\|\leq \|x\|\|y\| \qquad (x, y \in \mathscr{H}).
\end{align}

Davis \cite{Dav}, Joi\c{t}a \cite{Joi}, Ilisevi\'c--Varo\v{s}anec \cite{I.V}, Aramba\v{s}i\'{c}--Baki\'{c}--Moslehian \cite{A.B.M}
and Fujii--Fujii--Seo \cite{F.F.S} have investigated the Cauchy-–Schwarz inequality and its various reverses
in the framework of $C^*$-algebras and Hilbert $C^*$-modules.

Although Hilbert $C^*$-modules generalize Hilbert spaces some fundamental properties of Hilbert
spaces are no longer valid in Hilbert $C^*$-modules in their full generality.
For instance, norm-closed or even orthogonally closed submodules may not be orthogonal
summands, and an adjointable operator between Hilbert $C^*$-modules may have no polar
decomposition, cf. \cite{Wegge}. Therefore, when we are studying Hilbert $C^*$-modules, it is always of
interest under which conditions the results analogous to those for Hilbert spaces
can be reobtained, as well as which more general situations might appear.

In this paper, by using some ideas of \cite{C.K.K, Wat}, we introduce the generalized Cauchy--Schwarz
inequality for operators in $\mathcal{L}(\mathscr{H})$ and investigate various properties of operators
which satisfy the generalized Cauchy-Schwarz inequality.
In particular, we prove that if an operator $A\in\mathcal{L}(\mathscr{H})$ satisfies the generalized
Cauchy--Schwarz inequality, then $\mathcal{N}(A) = \mathcal{N}(A^2)$.
In addition, we show that if for $A$ the equality holds in the generalized Cauchy--Schwarz inequality,
then $A$ is cohyponormal. Among other things, when $A$
has the polar decomposition, we prove that $A$ is semi-hyponormal if and only if
$\big\|\langle Ax, y\rangle\big\| \leq \big\|{|A|}^{1/2}x\big\|\big\|{|A|}^{1/2}y\big\|$
for all $x, y \in\mathscr{H}$. Some other related results are also discussed.
\section{Main Results}
Throughout this section, $\mathscr{A}$ is a $C^*$-algebra and $\mathscr{H}$ is a Hilbert $\mathscr{A}$-module.
We start our work with the following definition.
\begin{definition}\label{D.2.1}
An operator $A\in \mathcal{L}(\mathscr{H})$ is said to satisfy the generalized Cauchy--Schwarz inequality
if there exists $\lambda \in (0, 1)$ such that
\begin{align}\label{I.D.2.1}
\|\langle Ax, y\rangle\|\leq (\|Ax\|\|y\|)^{\lambda}(\|Ay\|\|x\|)^{1 - \lambda} \qquad (x, y \in \mathscr{H}).
\end{align}
\end{definition}
\begin{notation}\label{N.2.3}
The collection of adjointable operators on $\mathscr{H}$
which satisfy the generalized Cauchy--Schwarz inequality is denoted by $GCSI(\mathscr{H})$.
\end{notation}
\begin{remark}\label{R.2.2}
When $A$ is the identity operator on $\mathscr{H}$,
the inequality in definition \ref{D.2.1} turns into \eqref{I.1.1}, that is,
we arrive at the usual Cauchy--Schwarz inequality in $\mathscr{H}$.
\end{remark}
\begin{remark}
Let $A\in GCSI(\mathscr{H})$. Then there exists $\lambda\in(0,1)$ such that \eqref{I.D.2.1} holds.
If $\lambda'\in (\lambda, 1)$, then \eqref{I.D.2.1} is also valid for $\lambda'$.
Indeed, there is  $\alpha\in (0,1)$ such that $\lambda'=\alpha \lambda+(1-\alpha)$. By \eqref{I.D.2.1} we get
\begin{align}\label{a11}
\|\langle Ax, y\rangle\|^\alpha\leq (\|Ax\|\|y\|)^{\alpha\lambda}(\|Ay\|\|x\|)^{\alpha(1 - \lambda)} \qquad (x, y \in \mathscr{H}).
\end{align}
Also by \eqref{I.1.1} we have
\begin{align}\label{a12}
\|\langle Ax,y\rangle \|^{(1-\alpha)}\leq (\|Ax\|\|y\|)^{(1-\alpha)} \qquad (x, y \in \mathscr{H}).
\end{align}
From \eqref{a11} and \eqref{a12} we obtain
\begin{align*}
\|\langle Ax,y\rangle\|\leq (\|Ax\|\|y\|)^{\alpha\lambda+(1-\alpha)}(\|Ay\|\|x\|)^{\alpha(1-\lambda)} \qquad (x, y \in \mathscr{H}).
\end{align*}
Now, since $\alpha(1-\lambda) = 1 - \lambda'$, by the above above inequality we have
\begin{align*}
\|\langle Ax, y\rangle\|\leq (\|Ax\|\|y\|)^{\lambda'}(\|Ay\|\|x\|)^{1 - \lambda'} \qquad (x, y \in \mathscr{H}).
\end{align*}
\end{remark}
\begin{remark}\label{R.2.2.5}
let $\mathscr{A} = \mathbb{C}$, $\mathscr{H} = \mathbb{C}^3$ and put
\begin{align*}
A = \left(\begin{array}{ccc}
1 & 0 & 0\\
0 & 0 & 1\\
0 & 0 & 0\\
\end{array}\right) \in \mathcal{L}(\mathscr{H}).
\end{align*}
Then simple computations show that $A\notin GCSI(\mathscr{H})$ (see also Theorem \ref{T.2.5}).
\end{remark}
For an arbitrary Hilbert $C^*$-module $\mathscr{H}$ there are many examples for operators in $GCSI(\mathscr{H})$.
\begin{proposition}\label{T.2.4}
Let $A\in \mathcal{L}(\mathscr H)$ be a normal. Then
\begin{align*}
\|\langle Ax, y\rangle\|\leq (\|Ax\|\|y\|)^{\lambda}(\|Ay\|\|x\|)^{1-\lambda} \qquad (x, y \in \mathscr{H}),
\end{align*}
for every $\lambda\in (0,1)$.
\end{proposition}
\begin{proof}
Let $\lambda\in (0,1)$. Since $A$ is normal, for every $z \in \mathscr{H}$, we have $\|A^*z\|=\|Az\|$.
Therefore,
\begin{align*}
\|\langle Ax, y\rangle\|=\|\langle x,A^*y\rangle\|\leq \|x\|\|A^*y\|=\|x\|\|Ay\|,
\end{align*}
for all $x, y \in \mathscr{H}$. Thus
\begin{align}\label{sc11}
\|\langle Ax,y\rangle\|^{1-\lambda}\leq (\|x\|\|Ay\|)^{1-\lambda} \qquad(x, y\in \mathscr H).
\end{align}
Also by \eqref{I.1.1} we have
\begin{align}\label{sc22}
\|\langle Ax,y\rangle\|^{\lambda}\leq (\|Ax\|\|y\|)^{\lambda} \qquad(x, y\in \mathscr H).
\end{align}
Utilizing \eqref{sc11} and \eqref{sc22}, we deduce the desired result.
\end{proof}
In the following proposition, we provide some evident properties of $GCSI(\mathscr{H})$
deduced from Definition \ref{D.2.1}.
\begin{proposition}\label{P.2.4.1}
The following statements hold.
\begin{itemize}
\item[(i)] $GCSI(\mathscr{H})$ is closed in the strong topology for adjointable operators.
\item[(ii)] If $A\in GCSI(\mathscr{H})$, then $\alpha A\in GCSI(\mathscr{H})$ for all $\alpha \in \mathbb{C}$.
\item[(iii)] If $A\in GCSI(\mathscr{H})$ is invertible, then $A^{-1}\in GCSI(\mathscr{H})$.
\end{itemize}
\end{proposition}
The following auxiliary lemmas are needed for our investigation.
\begin{lemma}\cite[Theorem 3.2]{Lan}\label{L.2.1}
Let $A\in \mathcal{L}(\mathscr{H})$. Then
\begin{itemize}
\item[(i)] $\mathcal{R}(A)$ is closed if and only if $\mathcal{R}(A^*)$ is closed, and in this case, $\mathcal{R}(A)$
and $\mathcal{R}(A^*)$ are orthogonally complemented with $\mathcal{R}(A) = {\mathcal{N}(A^*)}^{\perp}$ and
$\mathcal{R}(A^*) = {\mathcal{N}(A)}^{\perp}$.
\item[(ii)] $\mathcal{N}(A) = \mathcal{N}(|A|)$, $\mathcal{N}(A^*) = {\mathcal{R}(A)}^{\perp}$ and
${\mathcal{N}(A^*)}^{\perp} = {\mathcal{R}(A)}^{\perp\perp} \supseteq \overline{\mathcal{R}(A)}$.
\end{itemize}
\end{lemma}
\begin{lemma}\label{L.2.2}
Let $A\in \mathcal{L}(\mathscr{H})_{+}$. Then
$\mathcal{N}(A^t) = \mathcal{N}(A)$ for any $t> 0$.
\end{lemma}
\begin{proof}
For every $t\in (0, 1)$, it is proved in \cite[Lemma 2.2]{V.M.X} that $\mathcal{N}(A^t) = \mathcal{N}(A)$.
Now, let $t\geq 1$. Put $B = A^t$. Then $B\in \mathcal{L}(\mathscr{H})_{+}$ and so $\mathcal{N}\big(B^{1/t}\big) = \mathcal{N}(B)$.
Thus
\begin{align*}
\mathcal{N}(A) = \mathcal{N}\big(B^{1/t}\big) = \mathcal{N}(B) = \mathcal{N}(A^t).
\end{align*}
\end{proof}
Here we present one of our main results.
\begin{theorem}\label{T.2.5}
Let $A\in GCSI(\mathscr{H})$.
Then $\mathcal{N}(A) = \mathcal{N}(A^2)$.
\end{theorem}
\begin{proof}
Since $A\in GCSI(\mathscr{H})$, there exists $\lambda \in (0, 1)$ such that (\ref{I.D.2.1}) holds.
Now, assume that $z\in\mathcal{N}(A^2)$.
Thus $A^2z = 0$. Applying (\ref{I.D.2.1}) with $x= |A|^2z$ and $y = Az$, we get
\begin{align*}
\big\||A|^2z\big\|^2 = \big\|\langle A|A|^2z, Az\rangle\big\|
\leq \big(\big\|A|A|^2z\big\|\|Az\|\big)^{\lambda}\big(\|A^2z\|\big\||A|^2z\big\|\big)^{1 - \lambda} = 0.
\end{align*}
Hence $|A|^2z = 0$, or equivalently, $z\in \mathcal{N}(|A|^2)$.
Since, by Lemma \ref{L.2.2} and Lemma \ref{L.2.1} (ii), $\mathcal{N}(|A|^2) = \mathcal{N}(|A|) = \mathcal{N}(A)$ we reach $z\in \mathcal{N}(A)$.
Thus $\mathcal{N}(A^2) \subseteq \mathcal{N}(A)$.
Also, since $\mathcal{N}(A) \subseteq \mathcal{N}(A^2)$ we therefore conclude that $\mathcal{N}(A) = \mathcal{N}(A^2)$.
\end{proof}
\begin{corollary}\label{C.2.6.2}
Let $A^*\in GCSI(\mathscr{H})$.
Then ${\mathcal{R}(A)}^{\perp} = {\mathcal{R}(A^2)}^{\perp}$.
\end{corollary}
\begin{proof}
Since $A^*\in GCSI(\mathscr{H})$, by Theorem \ref{T.2.5} and Lemma \ref{L.2.1} (ii), we have
\begin{align*}
{\mathcal{R}(A)}^{\perp} = \mathcal{N}(A^*) = \mathcal{N}\big((A^*)^2\big) = \mathcal{N}\big((A^2)^*\big) = {\mathcal{R}(A^2)}^{\perp}.
\end{align*}
\end{proof}
\begin{remark}\label{R.2.7}
We remark that the converse of Theorem \ref{T.2.5} is not true, in general.
For example, let $A\,:\ell^2(\mathbb{Z})\longrightarrow \ell^2(\mathbb{Z})$
be the weighted bilateral shift defined by
$Ae_n = \alpha_{n}e_{n+1}$ for all $n\in \mathbb{Z}$, where
\begin{align*}
\alpha_n : =\left\{\begin{array}{ll}
1 & n\neq 2\\
\frac{1}{2} & n=2.\end{array}\right.
\end{align*}
Then it is easy to see that $\mathcal{N}(A) = \mathcal{N}(A^2) = \{0\}$, but $A\notin GCSI\big(\ell^2(\mathbb{Z})\big)$.
\end{remark}
\begin{remark}\label{R.2.8}
For an arbitrary operator $A\in GCSI(\mathscr{H})$
we have $\mathcal{N}(A) \subseteq \mathcal{N}(A^*)$.
Indeed, if $y\in\mathcal{N}(A)$, then
\begin{align*}
\|A^*y\|^2 = \|\langle AA^*y, y\rangle\| \leq
(\|AA^*y\|\|y\|)^{\lambda}(\|Ay\|\|A^*y\|)^{1 - \lambda} = 0,
\end{align*}
for some $\lambda \in (0, 1)$.
Hence $A^*y = 0$, or equivalently, $y\in \mathcal{N}(A^*)$.
\end{remark}
The following result is well-known (see \cite{Ka}).
\begin{lemma}\label{L.2.6.1}
Let $X$ and $Y$ be Banach spaces, and let $A\in \mathcal{B}(X, Y)$.
Suppose $M$ is a closed linear subspace of $Y$ such that $\mathcal{R}(A)\cap M =\{0\}$
and $\mathcal{R}(A) + M$ is closed in $Y$. Then $\mathcal{R}(A)$ is closed in $Y$.
\end{lemma}
\begin{theorem}\label{P.2.6}
Let $A\in GCSI(\mathscr{H})$.
If $\mathcal{R}(A) + \mathcal{N}(A)$ is closed, then $\mathcal{R}(A)$ is closed.
\end{theorem}
\begin{proof}
By Lemma \ref{L.2.6.1}, it suffices to show that $\mathcal{R}(A) \cap \mathcal{N}(A) =\{0\}$.
Since $A\in GCSI(\mathscr{H})$, there exists $\lambda \in (0, 1)$ such that (\ref{I.D.2.1}) holds.
Let $y = Ax$ for some $x\in \mathscr{H}$ such that $Ay=0$. Then
\begin{align*}
\|\langle y, y\rangle\| = \|\langle Ax, y\rangle\| \leq
(\|Ax\|\|y\|)^{\lambda}(\|Ay\|\|x\|)^{1 - \lambda} = 0,
\end{align*}
and hence $y = 0$.
\end{proof}
\begin{remark}\label{R.2.6.1}
Let $A\in GCSI(\mathscr{H})$.
If $\mathcal{R}(A) + \mathcal{N}(A)$ is closed, then, by Theorem \ref{P.2.6} and Lemma \ref{L.2.1} (i),
$\mathcal{R}(A^*)$ is orthogonally complemented with $\mathcal{R}(A^*) = {\mathcal{N}(A)}^{\perp}$.
\end{remark}
Let us quote a result from \cite{F.M.X}.
For any $C^*$-algebra $\mathscr{B}$, let $\mathscr{B}_+$ be the set of all positive elements of $\mathscr{B}$.
\begin{lemma}\cite[Lemma 2.3]{F.M.X}\label{L.2.11.1}
Let $\mathscr{B}$ be a $C^*$-algebra and $a, b \in\mathscr{B}_+$ be such that
$\|a^{1/2}c\| \leq \|b^{1/2}c\|$ for all $c \in\mathscr{B}_+$. Then $a \leq b$.
\end{lemma}
\begin{theorem}\label{Thno}
Let $A\in \mathcal{L}(\mathscr H)$. If the following conditions hold, then $A$ is normal.
\begin{itemize}
\item[(i)] There is $x\in\mathscr{H}$ with $Ax\neq 0$ such that for every $y\in \mathscr{H}$
the equality in \eqref{I.D.2.1} holds for some $\lambda\in (0, 1)$.
\item[(ii)] $\|\langle Au,v\rangle\|\leq \|Au\|\|v\|$ for all $u, v\in \mathscr{H}$.
\end{itemize}
\end{theorem}
\begin{proof}
By (i), we have
\begin{align}\label{I.11.T.2.12}
\|\langle Ax, y\rangle\| = (\|Ax\|\|y\|)^{\lambda}(\|Ay\|\|x\|)^{1 - \lambda} \qquad (y\in \mathscr{H}).
\end{align}
We will show that for any vector $z\in\mathscr{H}$
\begin{align}\label{I.22.T.2.12}
\|Az\| \leq \|A^*z\|.
\end{align}
Note that \eqref{I.22.T.2.12} holds if $z\in \mathcal{N}(A^*)$, since \eqref{I.11.T.2.12} gives that $Az=0$.
So, let us suppose $z\notin \mathcal{N}(A^*)$. By Remark \ref{R.2.8} we know that $z\notin \mathcal{N}(A)$.
From \eqref{I.11.T.2.12} it follows that
\begin{align*}\label{I.33.T.2.12}
\left(\frac{\|\langle Ax, z\rangle\|}{\|Ax\|\|z\|}\right)^{\lambda} \left(\frac{\|\langle Ax, z\rangle\|}{\|Az\|\|x\|}\right)^{1-\lambda} = 1.
\end{align*}
Thus $\frac{\|\langle Ax, z\rangle\|}{\|Az\|\|x\|}\geq 1$, because $\frac{\|\langle Ax, z\rangle\|}{\|Ax\|\|z\|} \leq 1$
by the Cauchy--Schwarz inequality in $\mathscr{H}$. Therefore,
\begin{align*}
\|Az\| \leq \frac{\|\langle Ax, z\rangle\|}{\|x\|} = \frac{\|\langle x, A^*z\rangle\|}{\|x\|} \leq \|A^*z\|.
\end{align*}
Now, let $w \in \mathscr{H}$. Put $a = \langle |A|^2w, w\rangle$ and $b = \langle |A^*|^2w, w\rangle$.
Then $a, b \in \mathscr{A}_+$ and for any $c\in\mathscr{A}_+$, by \eqref{I.22.T.2.12}, we have
\begin{align*}
\|a^{1/2}c\| &= \sqrt{\|cac\|} = \sqrt{\|c\langle |A|^2w, w\rangle c\|}
\\& = \sqrt{\|\langle A(wc), A(wc)\rangle\|} = \|A(wc)\| \leq \|A^*(wc)\| = \|b^{1/2}c\|.
\end{align*}
So, by Lemma \ref{L.2.11.1}, we obtain
\begin{align*}
\langle |A|^2w, w\rangle = a \leq b = \langle |A^*|^2w, w\rangle.
\end{align*}
This implies $|A|^2 \leq |A^*|^2$.
In the next we show that $|A^*|^2\leq |A|^2$.
By use of (ii) we know that
\begin{align*}
\|\langle u,A^*v\rangle\|= \|\langle Au, v\rangle\|\leq \|Au\|\|v\| \qquad (u, v\in \mathscr{H}).
\end{align*}
This gives that $\|A^*v\|\leq \|Av\|$ for every $v\in \mathscr H$. By the same way of proof in the above we get
$|A^*|^2\leq |A|^2$ and the proof is completed.
\end{proof}
By the proof of Theorem \ref{Thno} we get the following result.
\begin{theorem}\label{T.2.12}
Let $A\in \mathcal{L}(\mathscr{H})$. If there is $x\in \mathscr H$ with $Ax\ne 0$
such that for every $y\in \mathscr H$ the equality in \eqref{I.D.2.1} holds
for some $\lambda\in (0,1)$, then $A$ is cohyponormal.
\end{theorem}
The following example shows that the revers of Theorem \ref{T.2.12} is not true.
\begin{example}
Let $A\,:\ell^2(\mathbb{Z})\longrightarrow \ell^2(\mathbb{Z})$ be the unilateral shift operator. Then $A$ is cohyponormal.
Note there is no $x\in \ell^2(\mathbb{Z})$ with $Ax\neq 0$ such that the equality in \eqref{I.D.2.1} is not valid.
Indeed, let $x=\sum_{i=1}^\infty\alpha_ie_i$ where $\{e_i:i\in \mathbb{Z}\}$ is an orthonormal basis.
If the equality in \eqref{I.D.2.1} holds, then for every $k\in \mathbb{Z}$ there is $\lambda_k\in (0,1)$ such that
\begin{align*}
|\alpha_{k+1}|=\|\langle Ax,e_k\rangle\|=\left(\sum_{i=1}^\infty|\alpha_{i+1}|^2\right)^{\frac{\lambda_k}{2}}
\left(\sum_{i=1}^\infty|\alpha_i|^2\right)^{\frac{1-\lambda_k}{2}}\qquad(k=1,2,...).
\end{align*}
This yields that $\alpha_k=0$ for all $k\geq 2$ and so $x=\alpha_1e_1$. In this case we get $Ax=0$.
\end{example}
As an immediate consequence of Theorem \ref{T.2.12} we have the following result.
\begin{corollary}
Let $A\in \mathcal{L}(\mathscr{H})$ and let $x, x'\in \mathscr{H}$ with $Ax\neq 0$ and $A^*x'\neq0$.
Suppose that there are $\lambda,\lambda'\in (0,1)$ such that
\begin{align*}
\|\langle Ax,y\rangle\|=(\|Ax\|\|y\|)^\lambda(\|x\|\|Ay\|)^{1-\lambda}
\end{align*}
and
\begin{align*}
\|\langle A^*x',y\rangle\|=(\|A^*x'\|\|y\|)^{\lambda'}(\|x'\|\|A^*y\|)^{1-\lambda'} \qquad (y\in \mathscr{H}).
\end{align*}
Then $A$ is normal.
\end{corollary}
Recall that an element $U$ of $\mathcal{L}(\mathscr{H})$ is said to be a partial
isometry if $U^*U$ is a projection in $\mathcal{L}(\mathscr{H})$.
\begin{lemma}\cite[Proposition 15.3.7]{Wegge}\label{L.2.3}
Let $A\in \mathcal{L}(\mathscr{H})$. Then the following statements are equivalent:
\begin{itemize}
\item[(i)] $\mathscr{H} = \mathcal{N}(|A|)\oplus \overline{\mathcal{R}(|A|)}$ and $\mathscr{H} = \mathcal{N}(A^*)\oplus \overline{\mathcal{R}(A)}$.
\item[(ii)] Both $\overline{\mathcal{R}(|A|)}$ and $\overline{\mathcal{R}(A)}$ are orthogonally complemented.
\item[(iii)] $A$ has the polar decomposition $A = U|A|$, where $U\in\mathcal{L}(\mathscr{H})$ is a partial
isometry such that
\begin{align}
\mathcal{N}(U) = \mathcal{N}(A), \quad \mathcal{N}(U^*) = \mathcal{N}(A^*),
\quad \mathcal{R}(U) = \overline{\mathcal{R}(A)}, \quad \mathcal{R}(U^*) = \overline{\mathcal{R}(A^*)}.
\end{align}
\end{itemize}
\end{lemma}
The following lemma will be useful in the proof of the next result.
\begin{lemma}\cite[Lemma 3.12]{Liu-Luo-Xu}\label{L.2.4}
Let $A = U|A|$ be the polar decomposition of $A\in \mathcal{L}(\mathscr{H})$.
Then for any $t> 0$, the following statements are valid:
\begin{itemize}
\item[(i)] $U|A|^tU^* = (U|A|U^*)^t = |A^*|^t$.
\item[(ii)] $U|A|^t = |A^*|^tU$.
\item[(iii)] $U^*|A^*|^tU = (U^*|A^*|U)^t = |A|^t$.
\end{itemize}
\end{lemma}
\begin{theorem}\label{P.2.11}
Let $A\in \mathcal{L}(\mathscr{H})$ have the polar decomposition $A = U|A|$.
If $A$ is semi-hyponormal, then $A$ belongs to $GCSI(\mathscr{H})$ with $\lambda = 1/2$.
\end{theorem}
\begin{proof}
Let $x, y \in \mathscr{H}$. Since $|A^*|\leq |A|$, we get $0 \leq \langle |A^*|y, y\rangle \leq \langle |A|y, y\rangle$
and hence
\begin{align}\label{I.1.P.2.11}
\big\|\langle |A^*|y, y\rangle\big\| \leq \big\|\langle |A|y, y\rangle\big\|.
\end{align}
Therefore, by the Cauchy--Schwarz inequality, we have
\begin{align*}
\big\|\langle Ax, y\rangle\big\|^2 &= \big\|\langle |A|x, U^*y\rangle\big\|^2
= \big\|\langle |A|^{1/2}x, |A|^{1/2}U^*y\rangle\big\|^2
\\& \leq \big\|\langle |A|^{1/2}x, |A|^{1/2}x\rangle\big\| \big\|\langle |A|^{1/2}U^*y, |A|^{1/2}U^*y\rangle\big\|
\\& = \big\|\langle |A|x, x\rangle\big\| \big\|\langle U|A|U^*y, y\rangle\big\|
\\& = \big\|\langle |A|x, x\rangle\big\| \big\|\langle |A^*|y, y\rangle\big\| \qquad \qquad \qquad \qquad \big(\mbox{by Lemma \ref{L.2.4}~(i)}\big)
\\& \leq \big\|\langle |A|x, x\rangle\big\| \big\|\langle |A|y, y\rangle\big\| \qquad \qquad \qquad \qquad \,\, \big(\mbox{by}\,\,(\ref{I.1.P.2.11})\big)
\\& \leq \big\||A|x\big\|\|x\|\big\||A|y\big\|\|y\|
\\& = \big\|Ax\big\|\|x\|\big\|Ay\big\|\|y\| \qquad
\big(\mbox{since}\,\,\big\||A|z\big\| = \big\|Az\big\| \,\,\mbox{for every}\,\,z \in \mathscr{H}\big)
\\& = \left(\left(\big\|Ax\big\|\|y\|\right)^{1/2}\left(\big\|Ay\big\|\|x\|\right)^{1/2}\right)^2.
\end{align*}
Thus
\begin{align*}
\big\|\langle Ax, y\rangle\big\|\leq (\big\|Ax\big\|\|y\|)^{1/2}(\big\|Ay\big\|\|x\|)^{1/2}
\end{align*}
and so $A$ belongs to $GCSI(\mathscr{H})$ with $\lambda = 1/2$.
\end{proof}
Our next result reads as follows.
\begin{theorem}\label{L.2.15}
Let $A\in \mathcal{L}(\mathscr{H})$ have the polar
decomposition $A = U|A|$. If $A\in GCSI(\mathscr{H})$, then $A$ is paranormal.
\end{theorem}
\begin{proof}
Let $A\in GCSI(\mathscr{H})$. Therefore, there exists $\lambda \in (0, 1)$ such that (\ref{I.D.2.1}) holds.
We first show that for any vector $z \in \mathscr{H}$
\begin{align}\label{I.2.L.2.15}
\big\|AU^*z\big\|^2 \leq \big\|A^2U^*z\big\|\|U^*z\|.
\end{align}
If $z\in \mathcal{N}(U^*)$, then (\ref{I.2.L.2.15}) is trivially true. So, assume that $z\notin \mathcal{N}(U^*)$.
It follows from Lemma \ref{L.2.3} that $\mathcal{N}(A^*) = \mathcal{N}(U^*)$, so we get $z\notin \mathcal{N}(A^*)$.
Hence $\|A^*z\|>0$.
Therefore,
\begin{align*}
\big\|AU^*z\big\|^2 &= \big\|\langle AU^*z, AU^*z\rangle\big\|
\\& = \big\|\langle U|A|^2U^*z, z\rangle\big\|
\\& = \big\|\langle |A^*|^2z, z\rangle\big\| \qquad \qquad \qquad \qquad \big(\mbox{by Lemma \ref{L.2.4}~(i)}\big)
\\& = \big\||A^*|z\big\|^2 = \|A^*z\|^2.
\end{align*}
Thus
\begin{align}\label{I.3.L.2.15}
\big\|AU^*z\big\| = \big\||A^*|z\big\|> 0.
\end{align}
Again, by Lemma \ref{L.2.4}~(i), we have
\begin{align*}
\big\|A|A^*|z\big\|^2 &= \big\|\langle A|A^*|z, A|A^*|z\rangle\big\| = \big\|\langle AU|A|U^*z, AU|A|U^*z\rangle\big\|
\\& = \big\|\langle A^2U^*z, A^2U^*z\rangle\big\| = \big\|A^2U^*z\big\|^2,
\end{align*}
whence
\begin{align}\label{I.4.L.2.15}
\big\|A|A^*|z\big\| = \big\|A^2U^*z\big\|.
\end{align}
Utilizing Lemma \ref{L.2.4}~(i), (\ref{I.D.2.1}), (\ref{I.2.L.2.15}), (\ref{I.3.L.2.15}) and (\ref{I.4.L.2.15}), we have
\begin{align*}
\big\|AU^*z\big\|^2 &= \big\|U|A|U^*z\big\|^2
\\& = \big\|\langle U|A|U^*z, U|A|U^*z\rangle\big\|
\\& = \big\|\langle AU^*z, |A^*|z\rangle\big\|
\\& \leq (\big\|AU^*z\big\|\big\||A^*|z\big\|)^{\lambda}(\big\|A|A^*|z\big\|\big\|U^*z\big\|)^{(1 - \lambda)}
\\& = \big\|AU^*z\big\|^{2\lambda}(\big\|A^2U^*z\big\|\big\|U^*z\big\|)^{(1 - \lambda)},
\end{align*}
and hence
\begin{align*}
\big\|AU^*z\big\|^{2(1-\lambda)} \leq (\big\|A^2U^*z\big\|\big\|U^*z\big\|)^{(1 - \lambda)}.
\end{align*}
This ensures that
\begin{align}\label{I.5.L.2.15}
\big\|AU^*z\big\|^2 \leq \big\|A^2U^*z\big\|\big\|U^*z\big\|.
\end{align}
We need only to show that $\|Aw\|^2 \leq \|A^2w\|\|w\|$ for every $w \in \mathscr{H}$.
To this end, suppose that $w \in \mathscr{H}$.
Since $\mathscr{H} = \mathcal{N}(U)\oplus \mathcal{R}(U^*) = \mathcal{N}(A)\oplus \mathcal{R}(U^*)$,
there exist $v\in \mathcal{N}(A)$ and $z \in \mathscr{H}$ such that $w = v + U^*z$.
Then
\begin{align*}
\|w\|^2 = \|\langle v, v\rangle + \langle U^*z, U^*z\rangle\|
\geq \|\langle U^*z, U^*z\rangle\| = \|U^*z\|^2,
\end{align*}
and so
\begin{align}\label{I.6.L.2.15}
\|w\|\geq \|U^*z\|.
\end{align}
Employing (\ref{I.5.L.2.15}) and (\ref{I.6.L.2.15}) we obtain
\begin{align*}
\|Aw\|^2 = \|AU^*z\|^2 \leq \|A^2U^*z\|\|U^*z\| \leq \|A^2w\|\|w\|
\end{align*}
and we arrive at the desired result.
\end{proof}
In the following theorem, we get a characterization of semi-hyponormality of
adjointable operators.
\begin{theorem}\label{T.2.14}
Let $A\in \mathcal{L}(\mathscr{H})$ have the polar decomposition $A = U|A|$.
Then the following statements are equivalent:
\begin{itemize}
\item[(i)] $A$ is semi-hyponormal.
\item[(ii)] $\big\|\langle Ax, y\rangle\big\| \leq \big\|{|A|}^{1/2}x\big\|\big\|{|A|}^{1/2}y\big\|$ for all $x, y \in\mathscr{H}$.
\end{itemize}
\end{theorem}
\begin{proof}
(i)$\Rightarrow$(ii) The implication follows from the proof of Theorem \ref{P.2.11}.

(ii)$\Rightarrow$(i) Suppose (ii) holds.
We will show that for any vector $z\in\mathscr{H}$
\begin{align}\label{I.1.T.2.14}
\big\||A^*|^{1/2}z\big\| \leq \big\||A|^{1/2}z\big\|.
\end{align}
Note that (\ref{I.1.T.2.14}) holds if $z\in \mathcal{N}\big(|A^*|^{1/2}\big)$.
So, let us suppose $z\notin \mathcal{N}\big(|A^*|^{1/2}\big)$.
By letting $x = U^*z$ and $y = z$ in (ii), we have
\begin{align}\label{I.2.T.2.14}
\big\|\langle AU^*z, z\rangle\big\|^2 \leq \big\|\langle |A|U^*z, U^*z\rangle\big\|\big\|\langle |A|z, z\rangle\big\|.
\end{align}
Therefore,
\begin{align*}
\big\||A^*|^{1/2}z\big\|^4 &= \big\|\langle |A^*|z, z\rangle\big\|^2
\\& = \big\|\langle U|A|U^*z, z\rangle\big\|^2 \qquad \qquad \qquad \qquad \qquad \big(\mbox{by Lemma \ref{L.2.4}~(i)}\big)
\\& = \big\|\langle AU^*z, z\rangle\big\|^2
\\& \leq \big\|\langle |A|U^*z, U^*z\rangle\big\| \big\|\langle |A|z, z\rangle\big\|
\qquad \qquad \qquad \qquad \qquad \big(\mbox{by (\ref{I.2.T.2.14})}\big)
\\& = \big\|\langle U|A|U^*z, z\rangle\big\| \big\|\langle |A|z, z\rangle\big\|
\\& = \big\|\langle |A^*|z, z\rangle\big\| \big\|\langle |A|z, z\rangle\big\| \qquad \qquad \qquad \qquad \big(\mbox{by Lemma \ref{L.2.4}~(i)}\big)
\\& = \big\||A^*|^{1/2}z\big\|^2 \big\||A|^{1/2}z\big\|^2,
\end{align*}
which yields $\big\||A^*|^{1/2}z\big\|^2 \leq \big\||A|^{1/2}z\big\|^2$.
Hence $\big\||A^*|^{1/2}z\big\| \leq \big\||A|^{1/2}z\big\|$.
Utilizing a similar argument as in Theorem \ref{T.2.12}, we conclude that $|A^*| \leq |A|$.
Thus $A$ is semi-hyponormal.
\end{proof}
We finish this paper with the following result.
\begin{corollary}\label{C.2.15}
Let $A\in \mathcal{L}(\mathscr{H})$ have the polar decomposition.
If the equality in $GCSI(\mathscr{H})$ holds for $A^*$, then
\begin{align*}
\big\|\langle Ax, y\rangle\big\| \leq \big\|{|A|}^{1/2}x\big\|\big\|{|A|}^{1/2}y\big\| \qquad (x, y \in\mathscr{H}).
\end{align*}
\end{corollary}
\begin{proof}
Since the equality in $GCSI(\mathscr{H})$ holds for $A^*$, by Theorem \ref{T.2.12}, it follows that $A^*$ is cohyponormal.
Thus $|A^*|^2 \leq |A|^2$. Hence $|A^*| \leq |A|$, that is, $A$ is semi-hyponormal.
So, by Theorem \ref{T.2.14}, we deduce the desired result.
\end{proof}
\textbf{Acknowledgement.}
The author would like to thank Prof.~M.~S.~Moslehian, Prof.~Q.~Xu and Dr.~R.~Eskandari for their invaluable suggestions while writing this paper.
\bibliographystyle{amsplain}

\end{document}